\documentclass[11pt]{amsart}

\usepackage{amsmath,amssymb,amscd}
\usepackage[dvips]{graphics}

\newtheorem{theorem}{Theorem}[section]
\newtheorem{corollary}[theorem]{Corollary}

\newtheorem{lemma}[theorem]{Lemma}
\newtheorem{proposition}[theorem]{Proposition}
\newtheorem{remark}[theorem]{Remark}

\def\RR{\mathbb{R}}
\def\CC{\mathbb{C}}

\def\PP{\mathbb{P}}
\def\ZZ{\mathbb{Z}}

\newcommand{\cE}{{\mathcal E}}

\newcommand{\cL}{{\mathcal L}}
\newcommand{\cO}{{\mathcal O}}

\newcommand{\cF}{{\mathcal F}}

\def\GL{\operatorname{GL}}
\def\pf{\operatorname{Pf}}
\def\ker{\operatorname{Ker}}
\def\coker{\operatorname{Coker}}
\def\M{\operatorname{M}}

\def\Id{\operatorname{Id}}
\def\pic{\operatorname{Pic}}
\def\ext{\operatorname{Ext}}

\begin{document}

\title{Indecomposable Matrices Defining Plane Cubics}
\author{Anita Buckley}

\address{Department of Mathematics, University of Ljubljana,
Jadranska 19, 1000 Ljubljana, Slovenia.
{\rm e-mail:} {\it anita.buckley@fmf.uni-lj.si}}%

\begin{abstract}
In this article we find all (decomposable and indecomposable) $6\times 6$  linear determinantal representations of  Weierstrass cubics.  As a corollary we verify the Kippenhahn conjecture for $\M_6$.
\end{abstract}

\maketitle

\section{Introduction}
\label{introdef}

Let $C$ be an irreducible curve in $\CC\PP^2$ defined by a polynomial $F(x,y,z)$ of degree $3$.
Every smooth cubic can be brought by a projective change of coordinates \cite{mumford} into a Weierstrass form
$$F(x,y,z)= y z^2 - x(x-y)(x-\lambda y)= 0,$$
or equivalently
$$F(x,y,z)=-y z^2+x^3+\alpha x y^2+\beta y^3=0,$$
 for some  $\lambda \neq 0, 1$ and $\alpha,\beta\in\CC$.

We consider the following question. For given $C$ find 
 a linear matrix
$$A(x,y,z)=x\,A_x+y\,A_y+z\,A_z$$
such that 
$$\det A(x,y,z)=c\, F(x,y,z)^r,$$
where $A_x,A_y,A_z \in \M_{3r}$ and $0 \neq  c\in \CC$. Here $\M_{3r}$ is the algebra of all $3r\times 3r$ matrices over $\CC$. 

We call $A$ \textit{determinantal representation}  of $C$ of order $r$.  Two determinantal representations $A$ and $A'$ are \textit{equivalent} if there exist $X, Y\in\GL_{3r}(\CC)$ such that
$$A'=XAY.$$ 
We study self-adjoint representations $A=A^{\ast}$ modulo unitary equivalence $Y=X^{\ast}$ and
skew-symmetric representations $A=-A^t$ under $Y=X^t$ equivalence.
Obviously, equivalent determinantal representations define the same curve.
\textit{Pfaffian representation} is a representation of order 2 with all $6\times 6$ matrices being skew-symmetric.
Study of pfaffian representations is strongly related to determinantal representations:
every $3\times 3$ determinantal representation $A$
induces a \textit{decomposable pfaffian representation}
$$\left[ \begin{array}{cc}
0 & A \\
-A^t & 0
\end{array}\right].$$
Note that the equivalence relation is well defined since
$$\left[ \begin{array}{cc}
0 & X A Y \\
-(X A Y)^t & 0
\end{array}\right]=\left[ \begin{array}{cc}
X & 0 \\
0 & Y^t
\end{array}\right]\left[ \begin{array}{cc}
0 & A \\
-A^t & 0
\end{array}\right]\left[ \begin{array}{cc}
X^t & 0 \\
0 & Y
\end{array}\right].$$\\

The history of determinantal representations of order 1 goes back to the middle of the 19th century~\cite{grassman},~\cite{schur}. Dickson~\cite{dickson} classified hypersurfaces that can be represented as linear determinants.
In the last two decades determinantal representations of hypersurfaces again became extremely popular due to their use in semidefinite programming~\cite{vinni3}. Semidefinite programming feasibly uses  convex sets determined by  linear matrix inequalities  
$$\left\{ (x_1,\ldots,x_d)\in \RR^d\ :\ A_0+x_1 A_1+\cdots + x_d A_d\geq 0\right\}, $$
where $A_0,\ldots,A_d$ are  real symmetric or complex hermitian  matrices. A natural question in this framework is whether a  positive integer power of the polynomial has a determinantal representations.

All modern treatments of  determinantal representations of order 1 involve~\cite{cook},  the famous $1-1$ correspondence between
detreminantal representations of hypersurfaces and line bundles (i.e., points on
the Jacobian variety).
Analogously,
there is a one to one correspondence between linear pfaffian representations (up
to equivalence)  and rank 2 vector bundles (up to isomorphism) with
certain properties. These well known results are summed up in Beauville~\cite{beauville}. 

Elliptic curves are of tame representation type according to
Atiyah~\cite{atiyah}. In particular, on a given cubic curve the number
of indecomposable ACM bundles of rank $r$ with trivial determinant equals to the
number of $r$-torsion points.
Recently, Ravindra and Tripathi~\cite{ravindra} used these indecomposable vector bundles to prove the existence of  indecomposable determinantal representations of order greater than two.
 
Vinnikov~\cite{vinnikov2} explicitly parametrised $3\times 3$ determinantal representations of a Weierstrass cubic by the affine points on the same cubic.  In this paper we classify up to equivalence all  linear pfaffian representations of a Weierstrass cubic: the decomposable ones are parametrized by the affine points on the cubic, additionally there are three indecomposable representations arising from nontrivial extensions of even theta characteristics (i.e., line bundles corresponding to 2-torsion points). The classification is based on  Lancaster-Rodman canonical forms of matrix pairs~\cite{rodman}.
 More generally, any linear representation of order 2 is equivalent to either a  block linear matrix or its cokernel is a degree 0 indecomposable rank 2 bundle of Atiyah.

The main result of this paper is Theorem~\ref{newthree} containing the  construction of the three indecomposable pfaffian representations for a Weierstrass cubic. We outline similar constructions for indecomposable determinantal representations of order $r\geq 2$ corresponding to indecomposable vector bundles of rank $r$. These computations are an appendix to the paper of Ravindra and Tripathi~\cite{ravindra}. 
As  a corollary we verify the Kippenhahn conjecture for $\M_6$. 

The author wishes to thank G. V. Ravindra for his guidance and his patience, and to I. Klep for pointing out the connection between our constructions and  the conjecture of Kippenhahn.

\section{Determinantal representations of Vinnikov}
\label{detrepp}

Vinnikov~\cite{vinnikov2} found an explicit one to one correspondence between the $3\times 3$ determinantal representations
(up to equivalence) of $C$ and the points on an affine piece of $C$:

\begin{lemma}[\cite{vinnikov2}] \label{vinncubb} Every smooth cubic in $\PP^2$ can be brought into the Weierstrass form
$$F(x_0,x_1,x_2)=-x_1 x_2^2+x_0^3+\alpha x_0 x_1^2+\beta x_1^3.$$
A complete set of determinantal representations of $F$ is
\begin{equation}\label{eqvinn}
x_0 \Id+x_2 \left[\begin{array}{ccc} 0&1&0\\ 0&0&1\\ 0&0&0 \end{array}\right]
+x_1 \left[\begin{array}{ccc} \frac{t}{2}&s&\alpha+\frac{3}{4}t^2\\ 0&-t&-s\\ -1&0&\frac{t}{2} \end{array}\right],
\end{equation}
where $s^2=t^3+\alpha t+\beta.$ Note that the last equation is exactly the affine part $F(t,1,s)$.
\end{lemma}

We briefly repeat the proof as we will use similar ideas for $6 \times 6$ skew-symmetric determinantal representations.

\begin{proof}  Let $A(x_0,x_1,x_2)=x_0A_0+x_2A_2+x_1A_1$ be a representation of $F(x_0,x_1,x_2)$.  First we show that it is equivalent to a representation with 
$$A_0=\Id \mbox{ and } A_2=\left[\begin{array}{ccc} 0&1&0\\ 0&0&1\\ 0&0&0 \end{array}\right].$$ 
 Observe that  $A_0$ is invertible  since $\det A(1,0,0)\neq 0$. We can multiply $A(x_0,x_1,x_2)$ by $A_0^{-1}$ to obtain an equivalent representation with $A_0=\Id$. The characteristic polynomial of $A_2$ equals $\det A(x_0,0,-1)=x_0^3$ which implies that $A_2$ is nilpotent. The nonzero term $x_1 x_2^2$ in $F$ determines the order of nilpotency.
We remark that this is a Lancaster--Rodman canonical form for real
matrix pairs~\cite{rodman} .
Further  $\GL_{3}$ action  from left and right which preserves this canonical form (the first two matrices in the determinantal representation) reduces $A_1$ to the above.

\end{proof}

The above is an implementation of the classic Cook--Thomas correspondence~\cite{cook}: 
\begin{proposition}[\cite{beauville},Proposition 3.1.]\label{prop22}
Let $\cL$ be a line bundle of degree $0$ on $C$ with $H^0(C,\cL)=0$ . Then there exists a $3\times 3$ linear matrix $A$ such that $F=\det A$ and 
$$0 \rightarrow  \cO_{\PP^2}(-2)^3 \stackrel{A}{ \longrightarrow} \cO_{\PP^2}(-1)^3 \rightarrow  \cL \rightarrow 0. $$
Conversely, let $A$ be a $3\times 3$ linear matrix with   $F=\det A$. Then the cokernel of $A$ is a line bundle $\cL$ on $C$ of 
 degree $0$ with $H^0(C,\cL)=0$ .
\end{proposition}
The explicit formula in Lemma~\ref{vinncubb} is due to the fact that 
the set of line bundles of degree 0 with no global sections equals to $JC \backslash \cO_C$, where $JC=\pic^0C$ is the Jacobian of $C$. Recall that $JC$ equals to the curve itsef. Therefore $JC \backslash \cO_C$  can be parametrised by the affine points on the curve $C$.

Observe that if the cokernel of $A$ is $\cL$, the cokernel of $A^t$ is $\cL^{-1}$. This can also be seen directly from Lemma~\ref{vinncubb}, just
multiply (\ref{eqvinn}) by anti-identity.
Then 
\begin{equation}\label{equantidiag}
A=x_0 \left[\begin{array}{ccc} 0&0&1\\ 0&1&0\\ 1&0&0 \end{array}\right]+x_2 \left[\begin{array}{ccc} 0&1&0\\ 1&0&0\\ 0&0&0 \end{array}\right]
+x_1 \left[\begin{array}{ccc} \alpha+\frac{3}{4}t^2 & s &\frac{t}{2} \\ 
-s&-t&0\\ 
\frac{t}{2} &0&-1 \end{array}\right]\end{equation}
and $A^t$ correspond to the inverse points $(t,s)$ and $(t,-s)$ respectively on the affine part $s^2=t^3+\alpha t+\beta$ of $JC=C$. 
We can also conclude that the above $A$ is symmetric if and only if $s=0$. This implies that on $C$ there are three symmetric determinantal representations corresponding to three 2-torsion points $\cL\cong\cL^{-1}$ on $JC$. In the affine coordinates these three points are $(t,0)$, where  $t^3+\alpha t+\beta=0$.

\section{Pfaffian representations}
\label{seclast}
In this section we will implement the one to one correspondence between linear pfaffian representations and rank 2 vector bundles with certain properties, applying similar methods as in the proof of Lemma~\ref{vinncubb}. The correspondence is described in
\begin{proposition}[\cite{beauville},Proposition 5.1.]\label{BeauPfaff}
Let $\cE$ be a rank $2$ bundle on $C$ with $\det \cE= \cO_C$ and $H^0(C,\cE)=0$. Then there exists a $6\times 6$ linear skew--symmetric matrix $A$ such that $F=\pf A$ and 
$$0 \rightarrow  \cO_{\PP^2}(-2)^6 \stackrel{A}{ \longrightarrow} \cO_{\PP^2}(-1)^6 \rightarrow  \cE \rightarrow 0. $$
Conversely, let $A$ be a $6\times 6$ linear skew--symmetric matrix with   $F=\pf A$. Then the cokernel of $A$ is a rank 2 bundle $\cE$ on $C$ with the above properties .
\end{proposition}
Pfaffian representations are equivalent under the action
$$A\ \mapsto\ P\cdot A\cdot P^t,$$
where $P$ is an invertible $6\times 6$ constant matrix. By suitable choices of $P$ it is possible to reduce the number of parameters in $A$. In other words, we will reduce
the number of equivalent representations in each equivalence class.
The proof of Theorem~\ref{newthree} outlines an algorithm for finding all pfaffian representations (up to equivalence) of
$$C=\{(x,y,z)\in\PP^2 \ :\  y z^2 - x(x-y)(x-\lambda y)= 0\}. $$
For the sake of clearer notation we always write just the upper triangle of  skew--symmetric matrices.

\begin{theorem} \label{newthree} Let $C$ be a smooth cubic in the Weierstrass form
$$F(x,y,z)=y z^2 - x(x-y)(x-\lambda y).$$
A complete set of pfaffian representations of $F$ consists up to equivalence of three indecomposable representations and for the whole affine curve of decomposable representations:
$$
\substack{x\left[\!\!\!\!\begin{array}{cccccc}
 \substack{0} \!\!&\!\!  \substack{0} \!\!&\!\!
\substack{0} \!\!&\!\! \substack{0} \!\!&\!\!  \substack{0} \!\!&\!\!  \substack{1} \\
             \!\!&\!\! \substack{0} \!\!&\!\! \substack{0} \!\!&\!\! \substack{0} \!\!&\!\!  \substack{1} \!&\!\!\!  \substack{0} \\
             \!\!&\!\!                      \!\!&\!\! \substack{0} \!\!&\!\! \substack{1} \!\!&\!\!  \substack{0} \!\!&\!\!  \substack{0} \\
             \!\!&\!\!              \!\!&\!\!               \!\!&\!\!  \substack{0} \!\!&\!\!  \substack{0} \!\!&\!\!  \substack{0} \\
                          \!\!&\!\!               \!\!&\!\!
             \!\!&\!\!              \!\!&\!\!  \substack{0} \!\!&\!\!  \substack{0} \\
                         \!\!&\!\!               \!\!&\!\!
             \!\!&\!\!              \!\!&\!\!               \!\!&\!\!  \substack{0}
\end{array}\!\!\!\! \right] }+
\substack{z\left[\!\!\!\!\begin{array}{cccccc}
  \substack{0} \!\!&\!\!  \substack{0} \!\!&\!\! \substack{0} \!\!&\!\! \substack{0} \!\!&\!\!  \substack{1} \!\!&\!\!  \substack{0} \\
           \!\!&\!\!  \substack{0} \!\!&\!\! \substack{0} \!\!&\!\! \substack{1} \!\!&\!\!  \substack{0} \!&\!\!\!  \substack{0} \\
             \!\!&\!\!              \!\!&\!\!  \substack{0} \!\!&\!\! \substack{0} \!\!&\!\!  \substack{0} \!\!&\!\!  \substack{0} \\
                        \!\!&\!\!               \!\!&\!\!
             \!\!&\!\! \substack{0} \!\!&\!\!  \substack{0} \!\!&\!\!  \substack{0} \\
                          \!\!&\!\!               \!\!&\!\!
             \!\!&\!\!              \!\!&\!\!  \substack{0} \!\!&\!\!  \substack{0} \\
                        \!\!&\!\!               \!\!&\!\!
             \!\!&\!\!              \!\!&\!\!               \!\!&\!\!  \substack{0}
\end{array}\!\!\!\! \right] }+
\substack{y\substack{\left[ \!\!\!\! \begin{array}{cccccc}
0 \!\!&\!\! 1 \!\!&\!\! 0 \!\!&\!\! \frac{3  t^2 -  2 t (1+ \lambda) -  (1 - \lambda)^2}{4} \!\!&\!\! 0 \!\!&\!\!  \frac{t-1-\lambda}{2}  \\
\!\!&\!\! 0 \!\!&\!\! 0 \!\!&\!\! 0 \!\!&\!\! -\!t \!\!&\!\! 0 \\
\!\!&\!\!  \!\!&\!\! 0 \!\!&\!\! \frac{t-1-\lambda}{2} \!\!&\!\! 0 \!\!&\!\! -\!1  \\
 \!\!&\!\!  \!\!&\!\!  \!\!&\!\! 0 \!\!&\!\! 0 \!\!&\!\! 0 \\
 \!\!&\!\!  \!\!&\!\!  \!\!&\!\!  \!\!&\!\! 0 \!\!&\!\! 0  \\
  \!\!&\!\!  \!\!&\!\!  \!\!&\!\!  \!\!&\!\!  \!\!&\!\! 0
\end{array}\!\!\!\! \right] }} \mbox{ for } t=0,1,\lambda$$
and 
$$
\substack{x\left[\!\!\!\!\begin{array}{cccccc}
 \substack{0} \!\!&\!\!  \substack{0} \!\!&\!\!
\substack{0} \!\!&\!\! \substack{0} \!\!&\!\!  \substack{0} \!\!&\!\!  \substack{1} \\
             \!\!&\!\! \substack{0} \!\!&\!\! \substack{0} \!\!&\!\! \substack{0} \!\!&\!\!  \substack{1} \!&\!\!\!  \substack{0} \\
             \!\!&\!\!                      \!\!&\!\! \substack{0} \!\!&\!\! \substack{1} \!\!&\!\!  \substack{0} \!\!&\!\!  \substack{0} \\
             \!\!&\!\!              \!\!&\!\!               \!\!&\!\!  \substack{0} \!\!&\!\!  \substack{0} \!\!&\!\!  \substack{0} \\
                          \!\!&\!\!               \!\!&\!\!
             \!\!&\!\!              \!\!&\!\!  \substack{0} \!\!&\!\!  \substack{0} \\
                         \!\!&\!\!               \!\!&\!\!
             \!\!&\!\!              \!\!&\!\!               \!\!&\!\!  \substack{0}
\end{array}\!\!\!\! \right] }+
\substack{z\left[\!\!\!\!\begin{array}{cccccc}
  \substack{0} \!\!&\!\!  \substack{0} \!\!&\!\! \substack{0} \!\!&\!\! \substack{0} \!\!&\!\!  \substack{1} \!\!&\!\!  \substack{0} \\
           \!\!&\!\!  \substack{0} \!\!&\!\! \substack{0} \!\!&\!\! \substack{1} \!\!&\!\!  \substack{0} \!&\!\!\!  \substack{0} \\
             \!\!&\!\!              \!\!&\!\!  \substack{0} \!\!&\!\! \substack{0} \!\!&\!\!  \substack{0} \!\!&\!\!  \substack{0} \\
                        \!\!&\!\!               \!\!&\!\!
             \!\!&\!\! \substack{0} \!\!&\!\!  \substack{0} \!\!&\!\!  \substack{0} \\
                          \!\!&\!\!               \!\!&\!\!
             \!\!&\!\!              \!\!&\!\!  \substack{0} \!\!&\!\!  \substack{0} \\
                        \!\!&\!\!               \!\!&\!\!
             \!\!&\!\!              \!\!&\!\!               \!\!&\!\!  \substack{0}
\end{array}\!\!\!\! \right] }+
\substack{y\substack{\left[ \!\!\!\! \begin{array}{cccccc}
0 \!\!&\!\! 0 \!\!&\!\! 0 \!\!&\!\! \frac{3  t^2 -  2 t (1+ \lambda) -  (1 - \lambda)^2}{4} \!\!&\!\! s \!\!&\!\!  \frac{t-1-\lambda}{2}  \\
\!\!&\!\! 0 \!\!&\!\! 0 \!\!&\!\! -\!s \!\!&\!\! -\!t \!\!&\!\! 0 \\
\!\!&\!\!  \!\!&\!\! 0 \!\!&\!\! \frac{t-1-\lambda}{2} \!\!&\!\! 0 \!\!&\!\! -\!1  \\
 \!\!&\!\!  \!\!&\!\!  \!\!&\!\! 0 \!\!&\!\! 0 \!\!&\!\! 0 \\
 \!\!&\!\!  \!\!&\!\!  \!\!&\!\!  \!\!&\!\! 0 \!\!&\!\! 0  \\
  \!\!&\!\!  \!\!&\!\!  \!\!&\!\!  \!\!&\!\!  \!\!&\!\! 0
\end{array}\!\!\!\! \right], }}
$$
where $s^2  =t(t-1)(t-\lambda)$. Note that the last equation is exactly the affine part $F(t,1,s)$.
\end{theorem}

The proof will be 
based on Lancaster--Rodman canonical forms of matrix pairs~\cite{rodman}.
Let $A=x A_x+z A_z+y A_y$ be a pfaffian representation of $C$.
Observe that $A_x$ is invertible and $A_z$ nilpotent since $C$ is defined by $\pf A$ and contains $x^3$ term and no $z^3$ term.
Moreover, $y z^2$ determines the order of nilpotency. This determines the unique skew-symmetric canonical form~\cite[Theorem 5.1]{rodman} for the first two matrices. In other words,  every pfaffian representation of $C$ can be put into
the following form
$$
\substack{x\left[\!\!\!\!\begin{array}{cccccc}
 \substack{0} \!\!&\!\!  \substack{0} \!\!&\!\!
\substack{0} \!\!&\!\! \substack{0} \!\!&\!\!  \substack{0} \!\!&\!\!  \substack{1} \\
             \!\!&\!\! \substack{0} \!\!&\!\! \substack{0} \!\!&\!\! \substack{0} \!\!&\!\!  \substack{1} \!&\!\!\!  \substack{0} \\
             \!\!&\!\!                      \!\!&\!\! \substack{0} \!\!&\!\! \substack{1} \!\!&\!\!  \substack{0} \!\!&\!\!  \substack{0} \\
             \!\!&\!\!              \!\!&\!\!               \!\!&\!\!  \substack{0} \!\!&\!\!  \substack{0} \!\!&\!\!  \substack{0} \\
                          \!\!&\!\!               \!\!&\!\!
             \!\!&\!\!              \!\!&\!\!  \substack{0} \!\!&\!\!  \substack{0} \\
                         \!\!&\!\!               \!\!&\!\!
             \!\!&\!\!              \!\!&\!\!               \!\!&\!\!  \substack{0}
\end{array}\!\!\!\! \right] }+
\substack{z\left[\!\!\!\!\begin{array}{cccccc}
  \substack{0} \!\!&\!\!  \substack{0} \!\!&\!\! \substack{0} \!\!&\!\! \substack{0} \!\!&\!\!  \substack{1} \!\!&\!\!  \substack{0} \\
           \!\!&\!\!  \substack{0} \!\!&\!\! \substack{0} \!\!&\!\! \substack{1} \!\!&\!\!  \substack{0} \!&\!\!\!  \substack{0} \\
             \!\!&\!\!              \!\!&\!\!  \substack{0} \!\!&\!\! \substack{0} \!\!&\!\!  \substack{0} \!\!&\!\!  \substack{0} \\
                        \!\!&\!\!               \!\!&\!\!
             \!\!&\!\! \substack{0} \!\!&\!\!  \substack{0} \!\!&\!\!  \substack{0} \\
                          \!\!&\!\!               \!\!&\!\!
             \!\!&\!\!              \!\!&\!\!  \substack{0} \!\!&\!\!  \substack{0} \\
                        \!\!&\!\!               \!\!&\!\!
             \!\!&\!\!              \!\!&\!\!               \!\!&\!\!  \substack{0}
\end{array}\!\!\!\! \right] }+
\substack{y\substack{\left[ \!\!\!\! \begin{array}{cccccc}
0 \!\!&\!\! c_{12} \!\!&\!\! c_{13} \!\!&\!\! c_{14} \!\!&\!\! c_{15} \!\!&\!\! c_{16} \\
\!\!&\!\! 0 \!\!&\!\! c_{23} \!\!&\!\! c_{24} \!\!&\!\! c_{25} \!\!&\!\! c_{26} \\
\!\!&\!\!  \!\!&\!\! 0 \!\!&\!\! c_{34} \!\!&\!\! c_{35} \!\!&\!\! c_{36} \\
 \!\!&\!\!  \!\!&\!\!  \!\!&\!\! 0 \!\!&\!\! c_{45} \!\!&\!\! c_{46} \\
 \!\!&\!\!  \!\!&\!\!  \!\!&\!\!  \!\!&\!\! 0 \!\!&\!\! c_{56}  \\
  \!\!&\!\!  \!\!&\!\!  \!\!&\!\!  \!\!&\!\!  \!\!&\!\! 0
\end{array}\!\!\!\! \right]. }}
$$
Since $\pf A$ defines the equation of $C$, we get
\begin{eqnarray*}
c_{36}& =&-1 ,\\
c_{26}&=& -c_{35} ,\\
c_{25} &= &-1-\lambda-c_{16}-c_{34},\\
c_{14} &= & c_{16} + c_{16}^2 + c_{34} + c_{16} c_{34} + c_{34}^2 + 2 c_{24} c_{35} + c_{16} c_{35}^2 
- c_{34} c_{35}^2 -\\
&&
  c_{23} c_{45} - c_{13} c_{46} + c_{23} c_{35} c_{46} - c_{12} c_{56} + 
  c_{13} c_{35} c_{56} + \lambda (1 + c_{16}+ c_{34}), \\
c_{15} &= & -c_{24} - 
  c_{16} c_{35} + c_{34} c_{35} - c_{23} c_{46} - c_{13} c_{56}.
\end{eqnarray*}
There are $15-5$ parameters $c_{ij}$ left in the representation.
Additionally, the coefficient at $y^3$ equals
$c_{14} c_{26} c_{35} - c_{14} c_{25} c_{36} - c_{13} c_{26} c_{45} + c_{12} c_{36} c_{45} + 
   c_{16} (c_{25} c_{34} - c_{24} c_{35} + c_{23} c_{45}) + c_{13} c_{25} c_{46} - c_{12} c_{35} c_{46} - 
   c_{15} (c_{26} c_{34} - c_{24} c_{36} + c_{23} c_{46}) + c_{14} c_{23} c_{56} - c_{13} c_{24} c_{56} + 
   c_{12} c_{34} c_{56} =0$.

\begin{lemma} \label{lemreduce}
The action
$A\ \mapsto\ P\cdot A\cdot P^t$ preserves the canonical form of the first two matrices in the representation if and only if  
$P$ equals 
$$\left[ \begin{array}{cc}
P_1 & P_2 \\
P_3 & P_1^{-1}+P_3 P_1^{-1}P_2
\end{array}\right]\ \mbox{or}\ 
\left[ \begin{array}{cc}
P_2 & P_1 \\
-P_1^{-1}+P_3 P_1^{-1}P_2 & P_3
\end{array}\right]$$ 
where $P_1$ is invertible and $P_i$ are of the form
$$\left[ \begin{array}{ccc}
p_{i1} & p_{i2} & p_{i3} \\
0 & p_{i1} & p_{i2}  \\
0 & 0 & p_{i1}  
\end{array}\right],\ \ i=1,2,3.$$
\end{lemma}

\begin{proof}
Denote 
$$I=\left[ \begin{array}{ccc}
0 & 0 & 1  \\
0 & 1 & 0   \\
1 & 0 & 0 
\end{array}\right]\ \mbox{ and }\ N=\left[ \begin{array}{ccc}
0 & 1 & 0   \\
1 & 0 & 0  \\
0 & 0 & 0 
\end{array}\right].$$
We will need the following obvious observation, which can be proved directly by comparing matrix elements:\\
Let $Y,Y'$ be $6\times 6$  matrices for which 
$$Y .\left[ \begin{array}{cc}
0 & I \\
-I & 0
\end{array}\right]=\left[ \begin{array}{cc}
0 & I \\
-I & 0
\end{array}\right].Y'\ \mbox{ and }\  Y .\left[ \begin{array}{cc}
0 & N\\
-N & 0
\end{array}\right]=\left[ \begin{array}{cc}
0 & N \\
-N & 0
\end{array}\right].Y'\ \mbox{ hold.}$$
Then 
$Y=\left[ \begin{array}{cc}
Y_1 & Y_2 \\
Y_3 & Y_4
\end{array}\right]$ and
$Y'=\left[ \begin{array}{cc}
Y_4^t & -Y_3^t \\
-Y_2^t & Y_1^t
\end{array}\right],$ 
where 
$$Y_i=\left[ \begin{array}{ccc}
y_{i1} & y_{i2} & y_{i3} \\
 0 & y_{i1} & y_{i2}  \\
 0 & 0 &  y_{i1} 
\end{array}\right],\ \ i=1,2,3,4.$$
We call the specific form of the above Toeplitz matrices "$\triangle$ form".

Now we can find all invertible
$$P=\left[ \begin{array}{cc}
P_1 & P_2 \\
P_3 & P_4
\end{array}\right]$$ that satisfy
$$P.\left[ \begin{array}{cc}
0 & I \\
-I & 0
\end{array}\right].P^t=\left[ \begin{array}{cc}
0 & I \\
-I & 0
\end{array}\right]\ \mbox{ and }\ 
P.\left[ \begin{array}{cc}
0 & N \\
-N & 0
\end{array}\right].P^t=\left[ \begin{array}{cc}
0 & N \\
-N & 0
\end{array}\right].$$
By the above observation all $P_i$'s are of $\triangle$ form. Moreover, 
\begin{eqnarray*}
 p_{11} p_{41}-p_{21} p_{31} =1,\\
p_{22} p_{31} + p_{21} p_{32} - p_{12} p_{41} - p_{11} p_{42}=0,\\
p_{23} p_{31} + p_{22} p_{32} + p_{21} p_{33} - p_{13} p_{41} - p_{12} p_{42} - p_{11} p_{43}=0.
\end{eqnarray*}
In other words,
if $P_1$ is invertible then  $P_4=P_1^{-1}+P_3P_1^{-1}P_2$. 
The same way we see that $P_3=-P_2^{-1}+P_1P_2^{-1}P_4$ when $P_2$ is invertible.

Since $P$ is invertible and consists of $\triangle$ blocks, at least one of $P_1,P_2$ is also invertible. Note that 
$$\left[ \begin{array}{cc}
P_1 & P_2 \\
P_3 & P_4
\end{array}\right].\left[ \begin{array}{cc}
0 & -\Id \\
\Id & 0
\end{array}\right]=\left[ \begin{array}{cc}
P_2 & -P_1 \\
P_4 & -P_3
\end{array}\right]$$
exchanges $P_1$ and $P_2$ which finishes the proof.
\end{proof}

The action of Lemma~\ref{lemreduce} enables us to reduce the number of parameters $c_{ij}$. 
We can choose such $P$ that its action eliminates
\begin{eqnarray}\label{reduceq}
c_{13}= c_{23}  = c_{46} = c_{56} = 0,\ c_{35} = 0  \mbox{ and }  c_{16}=c_{34}.
\end{eqnarray}
Indeed, if we choose $p_{11}=1$, the above condtions determine $p_{12},p_{13}$ and $p_{22},p_{23},p_{32},p_{33}$:
\begin{eqnarray*} 
p_{32} & \rightarrow & c_{56} - c_{35} p_{31} + c_{56} p_{31} p_{21}, \\
p_{22} & \rightarrow & -c_{23} + c_{35} p_{21},\\
 p_{33} & \rightarrow &    \frac1 2 ((c_{16} - c_{34} + c_{35}^2 - c_{23} c_{56}) p_{31} + 2 c_{46} (1 + p_{31} p_{21})), \\
p_{23} & \rightarrow & \frac1 2 (-2 c_{13} + (-c_{16} + c{34} + c_{35}^2 - c_{23} c_{56}) p_{21}), \\
p_{12} & \rightarrow & -c_{35} + c_{56} p_{21},\\
p_{13} & \rightarrow & \frac1 2 (c_{16} - c_{34} + c_{35}^2 - c_{23} c_{56} + 2 c_{46} p_{21}).
\end{eqnarray*}

The relations among $c_{ij}$ then simplify to:
\begin{eqnarray}
c_{14}  &=&  3 c_{16}^2 + \lambda + 2 c_{16} (1 + \lambda) ,\nonumber \\ 
c_{24}&=&-c_{15}, \nonumber \\  
  0&=&c_{15}^2 - 8 c_{16}^3 - c_{12} c_{45} -\lambda - \lambda^2 - 
 8 c_{16}^2 (1 + \lambda) - 2 c_{16} (1 + 3 \lambda + \lambda^2). \label{1rel}
\end{eqnarray}

which leaves us with 4 parameters $c_{12},c_{45},c_{15},c_{16}$ and equation~(\ref{1rel}) connecting them:
$$\left[ \begin{array}{cccccc}
0 & \mathbf{c_{12}} & 0 &c_{14} &\mathbf{c_{15}}&\mathbf{c_{16}} \\
& 0&0 & -\!c_{15}& -\!1\!\!-\!\!\lambda\!\!-\!\!2c_{16} & 0 \\
&& 0 & c_{16} &0 & -\!1 \\
& && 0 & \mathbf{c_{45}} & 0 \\
 &  &&&0 & 0  \\
& &&&&0
\end{array} \right]. $$

It is easy to check that $A \mapsto P\cdot A\cdot P^t$ from Lemma~\ref{lemreduce} preserves all zeros and $-1$ in the above matrix if and only if  
$$P_i=\left[ \begin{array}{ccc}
 p_{i1} & 0 & 0  \\
  0 & p_{i1} & 0  \\
 0 & 0 &  p_{i1} 
\end{array}\right] \mbox{ for } i=1,2,3,4,\ \mbox{ together with }\  p_{11} p_{41}-p_{21} p_{31} =1.$$
We can use this "diagonal" action to make $c_{45}=0$ by chosing appropriate $p_{41}$ like in 
(\ref{reduceq}). When $c_{15}\neq 0$ we can furthermore make $c_{12}=0$ by  $p_{11}=p_{41}=1,  p_{31}=0, p_{21}=-c_{12}/2 c_{15}$. 
The only case left to consider is $c_{15}=0$. The action which keeps $c_{45}=0$ maps $c_{12}\mapsto c_{12} p_{11}^2$ where $p_{11}p_{41}=1$ and $p_{31}=0$. Thus either $c_{12}=0$ or we can make $c_{12}=1$.

In order to simplify notations even further, we introduce parameters $t$ and $s$ by 
$c_{16}=\frac1 2(t-1-\lambda)$ and $c_{15}=s$.
When $c_{45}=c_{12}=0$ the  matrix becomes
$$\left[ \begin{array}{cccccc}
0 & 0 & 0 &\frac{3  t^2 -  2 t (1+ \lambda) -  (1 - \lambda)^2}{4}
 &s& \frac{t-1-\lambda}{2} \\
& 0&0 & -\!s& -\!t& 0 \\
&& 0 & \frac{t-1-\lambda}{2} &0 & -\!1 \\
& && 0 &0 & 0 \\
 &  &&&0 & 0  \\
& &&&&0
\end{array} \right] $$
and relation (\ref{1rel}) in the new parameters equals 
$s^2  -t(t-1)(t-\lambda)=0$.

Additionally we get 
$$\left[ \begin{array}{cccccc}
0 & 1 & 0 &\frac{3  t^2 -  2 t (1+ \lambda) -  (1 - \lambda)^2}{4}
 &0& \frac{t-1-\lambda}{2} \\
& 0&0 & 0& -\!t& 0 \\
&& 0 & \frac{t-1-\lambda}{2} &0 & -\!1 \\
& && 0 &0 & 0 \\
 &  &&&0 & 0  \\
& &&&&0
\end{array} \right] $$ where $t$ is one of the three solutions of
$0= -t(t-1)(t-\lambda)$.

\begin{remark}\label{remkdec}\rm{The representations in Theorem~\ref{newthree} are non-equivalent 
to each other, since they are not connected by the action $A \mapsto P\cdot A\cdot P^t$. 
}\end{remark}

\section{A remark about the moduli space of rank 2 bundles}
\label{pfaffrep}

Rank $2$ bundles with trivial determinant and no sections lie in the open set 
$$M_C(2,\cO_C)\ \backslash \ \Theta_{2,\cO_C},$$
where $M_C(2,\cO_C)$ is the moduli space of semistable rank 2 bundles with determinant $\cO_C$ and 
$ \Theta_{2,\cO_C}=\{\cE\in M_C(2,\cO_C) \, : \, h^0(C,\cE)\neq 0 \}$.
In ~\cite[$\S$ 4]{beauville} and~\cite[$\S$ 6]{buckos} we can find that there are no stable bundles in  $M_C(2,\cO_C)$ and the unstable part consists of decomposable vector bundles of the form $\cL\oplus\cL^{-1}$ for $\cL $ in the Jacobian $J C$. Moreover, $\Theta_{2,\cO_C}=\{ \cO_C\oplus\cO_C \}$. This yields a $1-1$ correspondence between the points in $M_C(2,\cO_C)\ \backslash \ \Theta_{2,\cO_C}$ and the open subset of Kummer variety
$$\left(JC\ \backslash \ \{\cO_C \} \right)/_{\equiv},$$
where $\equiv$ is the involution $\cL \mapsto \cL^{-1}$.

Let $A$ and $-A^t$ be $3\times 3$ determinantal representations with cokernels $\cL$ and $\cL^{-1}$ respectively, like in Section~\ref{detrepp}. Obviously $\cL\oplus\cL^{-1}$ and $\cL^{-1}\oplus\cL$ are isomorphic rank 2 vector bundles. This is aligned with the corresponding decomposable pfaffian representations. Indeed, even though $A$ and $-A^t$ are not 
necessarily equivalent determinantal representations,
$$\left[ \begin{array}{cc}
0 & A \\
-A^t & 0
\end{array}\right]\ \ \mbox{ and }\ \ \left[ \begin{array}{cc}
0 & -A^t \\
A & 0
\end{array}\right]$$
are equivalent pfaffian representations since
$$\left[ \begin{array}{cc}
0 & \Id \\
\Id & 0
\end{array}\right]\left[ \begin{array}{cc}
0 & A \\
-A^t & 0
\end{array}\right]\left[ \begin{array}{cc}
0 & \Id \\
\Id & 0
\end{array}\right]= \left[ \begin{array}{cc}
0 & -A^t \\
A & 0
\end{array}\right].$$

\begin{remark}{\rm  
Each point in the the moduli space $M_C(2,\cO_C)$
corresponds to a decomposable bundle which induces a decomposable pfaffian representation. However, this does not imply that every rank 2  bundle on $C$ is decomposable. This is because the moduli space consists of S-equivalence classes rather than isomorphic bundles. 

Take for example the cokernel of a symmetric $3\times 3$ representation; one of the three cases where
$s=0$ in (\ref{equantidiag}). Then $\cL$ is a $2-$torsion point on $JC$, or in other words an even theta characteristic. For this line bundle $\cL\cong \cL^{-1}$, the direct sum $\cL\oplus \cL$ and the non-trivial extension of $\cL$ by $\cL$ represent the same point in the moduli space. These two bundles are clearly not isomorphic and can be realized as cokernels of matrices from Theorem~\ref{newthree},
$$
\substack{x\left[\!\!\!\!\begin{array}{cccccc}
 \substack{0} \!\!&\!\!  \substack{0} \!\!&\!\!
\substack{0} \!\!&\!\! \substack{0} \!\!&\!\!  \substack{0} \!\!&\!\!  \substack{1} \\
             \!\!&\!\! \substack{0} \!\!&\!\! \substack{0} \!\!&\!\! \substack{0} \!\!&\!\!  \substack{1} \!&\!\!\!  \substack{0} \\
             \!\!&\!\!                      \!\!&\!\! \substack{0} \!\!&\!\! \substack{1} \!\!&\!\!  \substack{0} \!\!&\!\!  \substack{0} \\
             \!\!&\!\!              \!\!&\!\!               \!\!&\!\!  \substack{0} \!\!&\!\!  \substack{0} \!\!&\!\!  \substack{0} \\
                          \!\!&\!\!               \!\!&\!\!
             \!\!&\!\!              \!\!&\!\!  \substack{0} \!\!&\!\!  \substack{0} \\
                         \!\!&\!\!               \!\!&\!\!
             \!\!&\!\!              \!\!&\!\!               \!\!&\!\!  \substack{0}
\end{array}\!\!\!\! \right] }+
\substack{z\left[\!\!\!\!\begin{array}{cccccc}
  \substack{0} \!\!&\!\!  \substack{0} \!\!&\!\! \substack{0} \!\!&\!\! \substack{0} \!\!&\!\!  \substack{1} \!\!&\!\!  \substack{0} \\
           \!\!&\!\!  \substack{0} \!\!&\!\! \substack{0} \!\!&\!\! \substack{1} \!\!&\!\!  \substack{0} \!&\!\!\!  \substack{0} \\
             \!\!&\!\!              \!\!&\!\!  \substack{0} \!\!&\!\! \substack{0} \!\!&\!\!  \substack{0} \!\!&\!\!  \substack{0} \\
                        \!\!&\!\!               \!\!&\!\!
             \!\!&\!\! \substack{0} \!\!&\!\!  \substack{0} \!\!&\!\!  \substack{0} \\
                          \!\!&\!\!               \!\!&\!\!
             \!\!&\!\!              \!\!&\!\!  \substack{0} \!\!&\!\!  \substack{0} \\
                        \!\!&\!\!               \!\!&\!\!
             \!\!&\!\!              \!\!&\!\!               \!\!&\!\!  \substack{0}
\end{array}\!\!\!\! \right] }+
\substack{y\substack{\left[ \!\!\!\! \begin{array}{cccccc}
0 \!\!&\!\! 0 \!\!&\!\! 0 \!\!&\!\! \frac{3  t^2 -  2 t (1+ \lambda) -  (1 - \lambda)^2}{4} \!\!&\!\! 0 \!\!&\!\!  \frac{t-1-\lambda}{2}  \\
\!\!&\!\! 0 \!\!&\!\! 0 \!\!&\!\! 0 \!\!&\!\! -\!t \!\!&\!\! 0 \\
\!\!&\!\!  \!\!&\!\! 0 \!\!&\!\! \frac{t-1-\lambda}{2} \!\!&\!\! 0 \!\!&\!\! -\!1  \\
 \!\!&\!\!  \!\!&\!\!  \!\!&\!\! 0 \!\!&\!\! 0 \!\!&\!\! 0 \\
 \!\!&\!\!  \!\!&\!\!  \!\!&\!\!  \!\!&\!\! 0 \!\!&\!\! 0  \\
  \!\!&\!\!  \!\!&\!\!  \!\!&\!\!  \!\!&\!\!  \!\!&\!\! 0
\end{array}\!\!\!\! \right] }}
$$
and
\begin{equation}\label{skewmat}
\substack{x\left[\!\!\!\!\begin{array}{cccccc}
 \substack{0} \!\!&\!\!  \substack{0} \!\!&\!\!
\substack{0} \!\!&\!\! \substack{0} \!\!&\!\!  \substack{0} \!\!&\!\!  \substack{1} \\
             \!\!&\!\! \substack{0} \!\!&\!\! \substack{0} \!\!&\!\! \substack{0} \!\!&\!\!  \substack{1} \!&\!\!\!  \substack{0} \\
             \!\!&\!\!                      \!\!&\!\! \substack{0} \!\!&\!\! \substack{1} \!\!&\!\!  \substack{0} \!\!&\!\!  \substack{0} \\
             \!\!&\!\!              \!\!&\!\!               \!\!&\!\!  \substack{0} \!\!&\!\!  \substack{0} \!\!&\!\!  \substack{0} \\
                          \!\!&\!\!               \!\!&\!\!
             \!\!&\!\!              \!\!&\!\!  \substack{0} \!\!&\!\!  \substack{0} \\
                         \!\!&\!\!               \!\!&\!\!
             \!\!&\!\!              \!\!&\!\!               \!\!&\!\!  \substack{0}
\end{array}\!\!\!\! \right] }+
\substack{z\left[\!\!\!\!\begin{array}{cccccc}
  \substack{0} \!\!&\!\!  \substack{0} \!\!&\!\! \substack{0} \!\!&\!\! \substack{0} \!\!&\!\!  \substack{1} \!\!&\!\!  \substack{0} \\
           \!\!&\!\!  \substack{0} \!\!&\!\! \substack{0} \!\!&\!\! \substack{1} \!\!&\!\!  \substack{0} \!&\!\!\!  \substack{0} \\
             \!\!&\!\!              \!\!&\!\!  \substack{0} \!\!&\!\! \substack{0} \!\!&\!\!  \substack{0} \!\!&\!\!  \substack{0} \\
                        \!\!&\!\!               \!\!&\!\!
             \!\!&\!\! \substack{0} \!\!&\!\!  \substack{0} \!\!&\!\!  \substack{0} \\
                          \!\!&\!\!               \!\!&\!\!
             \!\!&\!\!              \!\!&\!\!  \substack{0} \!\!&\!\!  \substack{0} \\
                        \!\!&\!\!               \!\!&\!\!
             \!\!&\!\!              \!\!&\!\!               \!\!&\!\!  \substack{0}
\end{array}\!\!\!\! \right] }+
\substack{y\substack{\left[ \!\!\!\! \begin{array}{cccccc}
0 \!\!&\!\! 1 \!\!&\!\! 0 \!\!&\!\! \frac{3  t^2 -  2 t (1+ \lambda) -  (1 - \lambda)^2}{4} \!\!&\!\! 0 \!\!&\!\!  \frac{t-1-\lambda}{2}  \\
\!\!&\!\! 0 \!\!&\!\! 0 \!\!&\!\! 0 \!\!&\!\! -\!t \!\!&\!\! 0 \\
\!\!&\!\!  \!\!&\!\! 0 \!\!&\!\! \frac{t-1-\lambda}{2} \!\!&\!\! 0 \!\!&\!\! -\!1  \\
 \!\!&\!\!  \!\!&\!\!  \!\!&\!\! 0 \!\!&\!\! 0 \!\!&\!\! 0 \\
 \!\!&\!\!  \!\!&\!\!  \!\!&\!\!  \!\!&\!\! 0 \!\!&\!\! 0  \\
  \!\!&\!\!  \!\!&\!\!  \!\!&\!\!  \!\!&\!\!  \!\!&\!\! 0
\end{array}\!\!\!\! \right] }},\end{equation}
for $t$ satisfying $0=t(t-1)(t-\lambda)$.
}\end{remark}

\section{Determinantal representations of order $\geq 2$}

Consider a linear matrix $A$  with $\det A=F^r$, where $F$ defines $C$.
Since $C$ is smooth, we will prove that the cokernel of $A$ is a vector bundle of rank $r$. Indeed, by~\cite[Theorem A]{beauville} the cokernel is an arithmetically Cohen-Macaulay (ACM) sheaf. Over a regular scheme any Cohen-Macaulay sheaf is locally free, thus it is a vector bundle. Furthermore, localizing at the generic point of $C$, we can locally write $A$ as a diagonal matrix $[F,\ldots,F,0\ldots,0]$, where $F$ occurs $r$ times.
The same result was obtained in~\cite{eisenbud} and~\cite{bhs} using purely algebraic methods for matrix factorizations of polynomials. 

Atiyah~\cite{atiyah} classified indecomposable vector bundles on $C$ in the following theorem.
\begin{theorem}[\cite{atiyah}, Theorem 5, Corollary 1] \label{atiyahthm}
Let $C$ be an elliptic curve.
\begin{itemize}
\item[(i)] For any $r>0$ there exists a unique (up to isomorphism) indecomposable vector bundle  $\cF_r$ with $h^0(\cF_r)=1$. Moreover, $\cF_r$ is defined inductively by the short exact sequence
$$0 \rightarrow \cO_C \rightarrow \cF_{r+1} \rightarrow \cF_r \rightarrow 0, \mbox{ where } \cF_1=\cO_C.$$ 
\item[(ii)] For any indecomposable rank $r$ bundle $\cE$ of degree 0, there exists a line bundle  $\cL$ such that 
$\cE\cong \cF_r\otimes \cL$ and $\cL^{\otimes r}=\det \cE$.
\item[(iii)] $\cF_r$ is self-dual,  $\cF_r\cong \cF^{\vee}_r$.
\end{itemize}
\end{theorem}

By knowing all possible cokernels of determinantal representations, we can  find (up to equivalence) all determinantal representations of  $F$~\cite[Proposition 1.11.]{beauville}. 
For $r=2$ this 1--1 correspondence gives the following corollary.

\begin{corollary}\label{clasrep6}
Let $A$ be a linear $6\times 6$ matrix with $\det A=F^2$. Then $A$ is either equivalent to a block matrix
\begin{equation}\label{usualblock}
\left[\begin{array}{cc}
A_1 & 0\\
0 & A_2
\end{array}\right],\end{equation}
or the cokernel of $A$ equals to a nontrivial extension of a degree 0 line bundle $\cL\neq O_C$ by itself. Recall that such $\cL$ is the cokernel of some $3\times 3$ determinantal representation of $C$.  
\end{corollary}
\begin{proof}
Denote the cokernel of $A$ by $\cE$. 
Consider the short exact sequence 
$$0 \rightarrow \cL_1 \rightarrow \cE \rightarrow \cL_2 \rightarrow 0,$$
where the line bundles $ \cL_1, \cL_2$ are cokernels of $3\times 3$ determinantal representations of $C$. By the Riemann-Roch formula all $ \cL_1, \cL_2$ and $\cE$ have degree 0. Then $\cE$ either splits 
and defines (\ref{usualblock}), or is a nontrivial extension if and only if 
$$\ext^1( \cL_2, \cL_1)\cong \ext^1( \cL_2 \otimes  \cL_1^{-1}, \cO_C)\cong H^1( C, \cL_2 \otimes  \cL_1^{-1})\cong H^0( C,\cL_1\otimes \cL_2^{-1})\neq 0.$$
The only line bundle on $C$ with nontrivial group of sections is $\cO_C$.
This means that $\cE$ is indecomposable if and only if $\cL_1\otimes \cL_2^{-1}= \cO_C$. This is consistent with Atiyah's Theorem~\ref{atiyahthm} where every indecomposable rank 2 bundle of degree 0 fits into the exact sequence
\begin{equation}\label{atiyasequ}0 \rightarrow \cL \rightarrow \cE \rightarrow \cL \rightarrow 0\end{equation}
for some line bundle $\cL$ with $\det\cE=\cL^{\otimes 2}$ .
 By~\cite[Proposition 1.11.]{beauville}, given a coherent  sheaf $\cE$ on $\PP^2$, there exists an exact sequence
$$0 \rightarrow  \cO_{\PP^2}(-2)^6 \stackrel{A}{ \longrightarrow} \cO_{\PP^2}(-1)^6 \rightarrow  \cE \rightarrow 0$$
if and only if $H^0(\PP^2,\cE)=H^1(\PP^2,\cE)=0$.
This gives the 1--1 correspondence between indecomposable rank two bundles $\cE$ with $\det\cE=\cL^{\otimes 2}$  for some $\cL\neq \cO_C$ of degree 0, and indecomposable determinantal representations $A$ with $\det A=F^2$ and $\coker A=\cE$. Every  $\cL\neq \cO_C$ of degree 0 form (\ref{atiyasequ}) is by Proposition~\ref{prop22} a cokernel  of some $3\times 3$ determinantal representation of $C$. Clearly, $H^0(C,\cL)=0$ and by Rieman--Roch $H^1(C,\cL)=0$, which implies $H^0(C,\cE)=H^1(C,\cE)=0$.
On the other hand, minimal free resolution is unique up to isomorphism by~\cite{eisenbudsyzy}.
\end{proof}

Consider again $A$ with $\det A=F^r$ and cokernel $\cE$.
In~\cite[Theorem B]{beauville} we find that when $\cE$ is additionally equipped with an
$\epsilon-$symmetric ($\epsilon=(-1)^{r-1}$) invertible form 
$$\cE\times \cE\rightarrow  O_C(\alpha) \mbox{ for some }\alpha\in\ZZ,$$ then $A$ can be taken to be $\epsilon-$symmetric. 
From here it follows that 
there are exactly three symmetric $3\times 3$ determinantal representations of $C$ with cokernels $\cL_1,\cL_2,\cL_3$, the three even theta characteristics (i.e., $2-$torsion points in $JC$). Indeed, these are the only nontrivial line bundles of degree 0 with $\cL^{\otimes 2}\cong O_C$.
In rank 2 case, Theorem~\ref{atiyahthm} shows that
 the only indecomposable rank 2 bundles on $C$ with determinant $\cO_C$ are the nontrivial extensions of theta characteristics $\cL_i$ by themselves.
Take an indecomposable pfaffian representation from Theorem~\ref{newthree}.
Its cokernel is indeed the nontrivial extension of a $2-$torsion point $\cL_i = \cL_i^{-1}$. 
The three non-block matrices in Theorem~\ref{newthree} are by the above the only indecomposable skew-symmetric matrices with determinant $F^2$. For $r=3$ Ravindra and Tripathi~\cite{ravindra} proved the existence of eight indecomposable $9\times 9$ determinantal representations of $C$. The corresponding  indecomposable cokernels are extensions of the nontrivial $3-$torsion points on $JC$ with themselves (i.e., the eight flexes on the affine Weierstrass cubic). In order to explicitly construct these determinantal representations, we would need to repeat the proof of Theorem~\ref{newthree} for $9\times 9$ matrices. We invite the interested reader to compute them as "extensions" of block matrices with three nonzero blocks $A$; here $A$ comes from Lemma~\ref{vinncubb}  with $(s,t)$ being one of the eight flexes on the affine Weierstrass cubic.

As another interesting corollary of Theorem~\ref{newthree}  we are able to verify the conjecture of Kippenhahn~\cite{kippenh} for $M_6(\CC)$, the algebra of $6\times 6$ matrices.
Shapiro~\cite{shapiro} showed that the conjecture holds for $n\leq 5$. Laffey~\cite{laffey} constructed a counterexample with $n=8$. Thus $n=6$ is the only remaining case.

\begin{corollary}\label{kippcor} The  conjecture of Kippenhahn as stated here is true for  $n=6$. 

Let $H,K$ be $6\times 6$ complex Hermitian matrices and $F$ a homogeneous polynomial defining a smooth cubic, such that $$\det(x H + y K - z Id)=F(x,y,z)^2.$$
Then $H$ and $K$ are simultaneously unitarily similar to  direct sums. This means that there exists an unitary matrix $U$ and matrices $H_1,H_2,K_1,K_2\in M_3(\CC)$ such that
$$ UHU^{\ast}=\left[\begin{array}{cc}H_1 & 0 \\ 0 & H_2\end{array}\right] \mbox { and }  UKU^{\ast}= \left[\begin{array}{cc}K_1 & 0 \\ 0 & K_2\end{array}\right].$$
\end{corollary}

\begin{remark} {\rm
Note that  $F$ in Corollary~\ref{kippcor} defines a real cubic curve and that $z Id-x H-y K$ is a definite determinantal representation of $F$. In the terminology of linear matrix inequalities, $F$ is a \textit{real zero polynomial} and $(0,0)$ lies inside the convex set  of points $\{ (x,y)\in\RR^2\, :\, Id-x H-y K\geq 0\}$ called \textit{spectrahedron}. Spectrahedron is bounded by the compact part of the curve. For more constructions of definite determinantal representations of polynomials we refer the reader to~\cite{netzer} and~\cite{quarez} and the references therein.
}\end{remark}

\begin{proof}
Every smooth real cubic can be brought into a Weierstrass form by a real change of coordinates
$$\left[ \begin{array}{c}x\\y\\z \end{array} \right]\mapsto P \left[ \begin{array}{c}x\\y\\z \end{array} \right] ,
\mbox{ for some } P \in\GL_{3}(\RR).$$
In the new coordinates we get
$$\det(x A_x + y A_y + z A_z)=\left( -y z^2+x^3+\alpha x y^2+\beta y^3 \right)^2,\mbox{ where }\alpha,\beta\in \RR.$$
Note that $A_x, A_y, A_z$ are real linear combinations of $H,K,\Id$ and therefore Hermitian. Since $z Id-x H-y K$ is definite,  $A=x A_x + y A_y + z A_z$  is also definite.
We showed that the cokernel of  $A$ is a rank 2 bundle. Assume first that the cokernel is decomposable $\cL_1\oplus\cL_2$.
Then $A$ is equivalent to a block matrix 
$\left[\begin{array}{cc}A_1 & 0 \\ 0 & A_2\end{array}\right]$, where $\cL_i$ is the cokernel of $A_i$.  In particular, each line bundle $\cL_i$ satisfies the conditions in~\cite[Theorem 7]{vinnikov2}, therefore $A_i$ can be  brought by~\cite[Theorem 8]{vinni1} into one of the two self-adjoint forms     
\begin{equation}\label{equselfaj}
\pm\left( x \left[\begin{array}{ccc} 0&0&1\\ 0&1&0\\ 1&0&0 \end{array}\right]+z \left[\begin{array}{ccc} 0&1&0\\ 1&0&0\\ 0&0&0 \end{array}\right]
+y \left[\begin{array}{ccc} \alpha+\frac{3}{4}t_i^2 & {i\mkern1mu} s_i &\frac{t_i}{2} \\ 
-{i\mkern1mu} s_i&-t_i&0\\ 
\frac{t_i}{2} &0&-1 \end{array}\right]\right),
\end{equation} 
where $(s_i,t_i)\in\RR^2$ satisfy $-s_i^2=t_i^3+\alpha t_i+\beta$. Moreover, $A_i$ is unitarily equivalent to one of the above forms. It was shown in~\cite{vinni1}, that definite self-adjoint determinantal representations of $C$ are exactly those corresponding to the points $(s,t)$ in the compact part of $C(\RR)$. This self-adjoint canonical form can be obtained also directly, using  Lancaster--Rodman canonical forms for real
matrix pairs~\cite{rodman}, like we did in Section~\ref{detrepp}. Indeed, just replace $s$ by $ {i\mkern1mu} s$ in (\ref{equantidiag}).

Next we assume that the cokernel $\cE$ of $A$ is indecomposable. In Corollary~\ref{clasrep6} we showed that $\det \cE=\cL^{\otimes 2}$ for some line bundle $\cL\neq O_C$ of degree 0. 
Since $A$ is self-adjoint, $\cL$ is isomorphic to the cokernel of a self-adjoint representation  in~(\ref{equselfaj}),  $\cL_{s,t}$ where $ (s,t)\in\RR^2$. Note that $\cL_{s,t}^{-1}\cong \cL_{-s,t}$. 
Since $A^{\ast} (x,y,z)=A(\bar{x},\bar{y},\bar{z})$ and always holds $\cE\cong \cE^{\vee}\otimes \det\cE$, we conclude that $\cE^{\vee}=\ker A$ is a nontrivial extension of $\cL_{-s,t}$.

Analogous computations as in the proof of Theorem~\ref{newthree} show that $A$ is unitarily equivalent to
a self-adjoint representation
$$
\substack{x\left[\!\!\!\!\begin{array}{cccccc}
 \substack{0} \!\!&\!\!  \substack{0} \!\!&\!\!
\substack{0} \!\!&\!\! \substack{0} \!\!&\!\!  \substack{0} \!\!&\!\!  \substack{1} \\
             \!\!&\!\! \substack{0} \!\!&\!\! \substack{0} \!\!&\!\! \substack{0} \!\!&\!\!  \substack{1} \!&\!\!\!  \substack{0} \\
             \!\!&\!\!                      \!\!&\!\! \substack{0} \!\!&\!\! \substack{1} \!\!&\!\!  \substack{0} \!\!&\!\!  \substack{0} \\
             \!\!&\!\!              \!\!&\!\!               \!\!&\!\!  \substack{0} \!\!&\!\!  \substack{0} \!\!&\!\!  \substack{0} \\
                          \!\!&\!\!               \!\!&\!\!
             \!\!&\!\!              \!\!&\!\!  \substack{0} \!\!&\!\!  \substack{0} \\
                         \!\!&\!\!               \!\!&\!\!
             \!\!&\!\!              \!\!&\!\!               \!\!&\!\!  \substack{0}
\end{array}\!\!\!\! \right] }+
\substack{z\left[\!\!\!\!\begin{array}{cccccc}
  \substack{0} \!\!&\!\!  \substack{0} \!\!&\!\! \substack{0} \!\!&\!\! \substack{0} \!\!&\!\!  \substack{1} \!\!&\!\!  \substack{0} \\
           \!\!&\!\!  \substack{0} \!\!&\!\! \substack{0} \!\!&\!\! \substack{1} \!\!&\!\!  \substack{0} \!&\!\!\!  \substack{0} \\
             \!\!&\!\!              \!\!&\!\!  \substack{0} \!\!&\!\! \substack{0} \!\!&\!\!  \substack{0} \!\!&\!\!  \substack{0} \\
                        \!\!&\!\!               \!\!&\!\!
             \!\!&\!\! \substack{0} \!\!&\!\!  \substack{0} \!\!&\!\!  \substack{0} \\
                          \!\!&\!\!               \!\!&\!\!
             \!\!&\!\!              \!\!&\!\!  \substack{0} \!\!&\!\!  \substack{0} \\
                        \!\!&\!\!               \!\!&\!\!
             \!\!&\!\!              \!\!&\!\!               \!\!&\!\!  \substack{0}
\end{array}\!\!\!\! \right] }+
\substack{y\substack{\left[ \!\!\!\! \begin{array}{cccccc}
0 \!\!&\!\! \ast \!\!&\!\! \ast \!\!&\!\! \frac{3  t^2 -  2 t (1+ \lambda) -  (1 - \lambda)^2}{4 } \!\!&\!\! {i\mkern1mu} s \!\!&\!\!  \frac{t-1-\lambda}{2 }  \\
\!\!&\!\! 0 \!\!&\!\! \ast \!\!&\!\! - {i\mkern1mu} s \!\!&\!\! -\! t \!\!&\!\! 0 \\
\!\!&\!\!  \!\!&\!\! 0 \!\!&\!\! \frac{t-1-\lambda}{2 } \!\!&\!\! 0 \!\!&\!\! -\! 1 \\
 \!\!&\!\!  \!\!&\!\!  \!\!&\!\! 0 \!\!&\!\! 0 \!\!&\!\! 0 \\
 \!\!&\!\!  \!\!&\!\!  \!\!&\!\!  \!\!&\!\! 0 \!\!&\!\! 0  \\
  \!\!&\!\!  \!\!&\!\!  \!\!&\!\!  \!\!&\!\!  \!\!&\!\! 0
\end{array}\!\!\!\! \right] }} \mbox{ for } (s,t)\in \RR^2,$$
or ${i\mkern1mu} $ times the skew-symmetric representation~(\ref{skewmat})
$$
\substack{x\left[\!\!\!\!\begin{array}{cccccc}
 \substack{0} \!\!&\!\!  \substack{0} \!\!&\!\!
\substack{0} \!\!&\!\! \substack{0} \!\!&\!\!  \substack{0} \!\!&\!\!  \substack{i} \\
             \!\!&\!\! \substack{0} \!\!&\!\! \substack{0} \!\!&\!\! \substack{0} \!\!&\!\!  \substack{i} \!&\!\!\!  \substack{0} \\
             \!\!&\!\!                      \!\!&\!\! \substack{0} \!\!&\!\! \substack{i} \!\!&\!\!  \substack{0} \!\!&\!\!  \substack{0} \\
             \!\!&\!\!              \!\!&\!\!               \!\!&\!\!  \substack{0} \!\!&\!\!  \substack{0} \!\!&\!\!  \substack{0} \\
                          \!\!&\!\!               \!\!&\!\!
             \!\!&\!\!              \!\!&\!\!  \substack{0} \!\!&\!\!  \substack{0} \\
                         \!\!&\!\!               \!\!&\!\!
             \!\!&\!\!              \!\!&\!\!               \!\!&\!\!  \substack{0}
\end{array}\!\!\!\! \right] }+
\substack{z\left[\!\!\!\!\begin{array}{cccccc}
  \substack{0} \!\!&\!\!  \substack{0} \!\!&\!\! \substack{0} \!\!&\!\! \substack{0} \!\!&\!\!  \substack{i} \!\!&\!\!  \substack{0} \\
           \!\!&\!\!  \substack{0} \!\!&\!\! \substack{0} \!\!&\!\! \substack{i} \!\!&\!\!  \substack{0} \!&\!\!\!  \substack{0} \\
             \!\!&\!\!              \!\!&\!\!  \substack{0} \!\!&\!\! \substack{0} \!\!&\!\!  \substack{0} \!\!&\!\!  \substack{0} \\
                        \!\!&\!\!               \!\!&\!\!
             \!\!&\!\! \substack{0} \!\!&\!\!  \substack{0} \!\!&\!\!  \substack{0} \\
                          \!\!&\!\!               \!\!&\!\!
             \!\!&\!\!              \!\!&\!\!  \substack{0} \!\!&\!\!  \substack{0} \\
                        \!\!&\!\!               \!\!&\!\!
             \!\!&\!\!              \!\!&\!\!               \!\!&\!\!  \substack{0}
\end{array}\!\!\!\! \right] }+
\substack{y\substack{\left[ \!\!\!\! \begin{array}{cccccc}
0 \!\!&\!\! i \!\!&\!\! 0 \!\!&\!\! \frac{3  t^2 -  2 t (1+ \lambda) -  (1 - \lambda)^2}{-4 i} \!\!&\!\! 0 \!\!&\!\!  \frac{t-1-\lambda}{-2 i}  \\
\!\!&\!\! 0 \!\!&\!\! 0 \!\!&\!\! 0 \!\!&\!\! -\! i t \!\!&\!\! 0 \\
\!\!&\!\!  \!\!&\!\! 0 \!\!&\!\! \frac{t-1-\lambda}{-2 i} \!\!&\!\! 0 \!\!&\!\! -\!i  \\
 \!\!&\!\!  \!\!&\!\!  \!\!&\!\! 0 \!\!&\!\! 0 \!\!&\!\! 0 \\
 \!\!&\!\!  \!\!&\!\!  \!\!&\!\!  \!\!&\!\! 0 \!\!&\!\! 0  \\
  \!\!&\!\!  \!\!&\!\!  \!\!&\!\!  \!\!&\!\!  \!\!&\!\! 0
\end{array}\!\!\!\! \right] }} \mbox{ for } t=0,1,\lambda.$$
It is easy to check that non of these are definite and can thus not provide a counterexample to Kippenhahn's conjecture.  
\end{proof}

\end{document}